\documentclass[11pt,english]{article} 
\usepackage[utf8]{inputenc}
\usepackage[T1]{fontenc}
\usepackage{lmodern}
\usepackage[a4paper]{geometry}
\usepackage{babel}
\usepackage{cleveref}

\usepackage{nag}

\usepackage{mathtools}
\usepackage{amsthm}
\usepackage{amssymb}
\usepackage{graphicx}
\usepackage{caption}
\usepackage{enumitem}

                \newcommand{\R}{\mathbb R}

\newcommand{\D}{\partial}               
\newcommand{\s}{\sigma}

\newcommand{\sph}{\mathbb S}                 
               
     \def\a{\alpha}

               \newcommand{\A}{{\mathcal A}}  
            \newcommand{\RP}{{\mathbb R\mathrm P}}
\newcommand{\ind}{\mathrm{ind}}        \newcommand{\sign}{\mathrm{sign}}
\newcommand{\x}{{\tt x}}                \newcommand{\y}{{\tt y}}
\newcommand{\z}{{\tt z}}

\begin{document}

\captionsetup[figure]{labelfont={bf},labelformat={default},labelsep=period,name={\small Fig.}}


\numberwithin{equation}{section}                
\theoremstyle{plain}
\newtheorem{theorem}{\bf Theorem}
\newtheorem*{theorem*}{\bf Theorem}
\newtheorem{teorema}{\bf Theorem}
\def\theteorema{\Roman{teorema}}
\newtheorem{lemma}{\bf Lemma}
\newtheorem*{lemma*}{\bf Lemma}
\newtheorem*{mlemma}{\bf Main Lemma}

\newtheorem{proposition}{\bf Proposition}
\newtheorem*{proposition*}{\bf Proposition}
\newtheorem{corollary}{\bf Corollary}
\newtheorem*{corollary*}{\bf Corollary}
\newtheorem*{conjecture}{\bf Conjecture}
\newtheorem*{fact*}{\bf Fact}
\newtheorem{property}{\bf Property}


\theoremstyle{definition}
\newtheorem{definition}{\bf Definition}
\newtheorem*{definition*}{\bf Definition}

\newtheorem{example}{\bf Example}
\newtheorem*{example*}{\bf Example}
\newtheorem{Example}{\bf Example}[section]
\theoremstyle{remark}
\newtheorem*{remark*}{\small\sc Remark}
\newtheorem{remark}{\bf Remark}
\newtheorem*{problem}{\bf Problem}


\title{On Projective Umbilics: a Geometric Invariant and an Index}

\author{Ricardo \sc Uribe-Vargas \\ {\small Published in Journal of Singularities \textbf{17} (2018), 81-90. DOI: 10.5427/jsing.2018.17e}}
\date\empty                     
\maketitle

\noindent 
\textbf{Abstract}\,
  We define a geometric invariant and an index ($+1$ or $-1$)
  for projective umbilics of smooth surfaces. We prove that {\itshape the sum of the
    indices of the projective umbilics inside a connected component $H$ of the hyperbolic domain remains
    constant in any $1$-parameter family of surfaces if the topological type of $H$
    does not change.} We prove the same statement for any connected component $E$ of the elliptic domain. 
We give formulas for the invariant and for the index which do not depend on any normal form. 
\smallskip




\section{Introduction}

The points at which a smooth plane curve has higher contact with its osculating circle (\textit{vertices})
or with its osculating conic (\textit{sextactic points}) have been extensively studied. Recent
works relate this classical subject of differential geometry to singularity theory, and to contact and
symplectic geometry (cf. \cite{Arnold-RPCS, Arnold-TPTAC, Chekanov-Pushkar, Goryunov84, Dima, Uribeinvariant, Toru}).

For a smooth surface of $\R^3$ the analogues of the vertices, called  \textit{umbilics}, are the points
at which a sphere has higher contact (than usual) with the surface.

\textit{Projective umbilics} are the analogous of sextactic points for smooth surfaces of $\RP^3$:
the points where the surface is approximated by a quadric up to order
$3$. (Projective umbilics are also called \textit{quadratic points} (see \cite{Ovsienko-Tabachnikov}).) 

Although projective umbilics have been studied in classical literature (cf. \cite{Cayley, Salmon, Wilczynski}),
the few known global results about them are rather recent. 
Let us mention two examples\,: 
{\itshape a convex smooth surface of $\RP^3$ has at least $6$ projective umbilics} \cite{Dima};
{\itshape if a generic smooth surface of $\RP^3$ contains a hyperbolic
  disc bounded by a Jordan parabolic curve, then there exists an odd number of projective umbilics inside this
  disc} (and hence at least one) \cite{Uribeinvariant}. 
We mention other recent results in Section \ref{remarks}. 

Tabachnikov and Ovsienko stated a (still open) conjecture \cite{Ovsienko-Tabachnikov}:
{\itshape the least number of projective umbilics on a generic compact smooth hyperbolic surface is 
$8$.}\footnote{An analogue conjecture, but for surfaces of $\sph^3$, was named 
``\textit{Ovsienko-Tabachnikov conjecture}'' 
in \cite{Garcia-Freitas} and negatively solved. However, the original 
Ovsienko-Tabachnikov conjecture was not solved in \cite{Garcia-Freitas}.}
\medskip

\noindent
\textbf{\small Hyperbonodes - Ellipnodes}. 
To make it short, we shall call \textit{hyperbonodes} the hyperbolic projective umbilics and \textit{ellipnodes}
the elliptic ones.
\medskip

To present our results, we will describe some features of generic smooth surfaces of
$\RP^3$ (and of $\R^3$) and will give another characterisation of hyperbonodes:  
\smallskip

\noindent
\textbf{\small Contact with lines}. 
The points of a surface in $3$-space are characterised by \textit{asymptotic lines}
(tangent lines of the surface with more than $2$-point contact).
A point is called hyperbolic (resp. elliptic) if there is two distinct asymptotic lines
(resp. no asymptotic line). The \textit{parabolic curve} consists of the points where
there is a unique (but double) asymptotic line. The \textit{flecnodal curve} is the locus of
hyperbolic points where an asymptotic line admits more than $3$-point contact with the surface.

We distinguish two branches of the flecnodal curve\,: ``left'' and ``right'', defined below.  
We represent the right branch of the flecnodal curve and the right asymptotic lines in black, 
and the left branch and asymptotic lines in white. The elliptic domain is represented in white. 

A \textit{hyperbonode} is a point of intersection of the left and right branches of the flecnodal
curve (Fig.\ref{hyperbonode}).
A \textit{biflecnode} is a point of the flecnodal curve at which one asymptotic
tangent line has at least $5$-point contact with $S$ - {\itshape at a biflecnode that
  asymptotic line is tangent to the corresponding flecnodal curve}
(Fig.\ref{biflecnodes}).

\begin{figure}[!ht]
\centering
\begin{minipage}[t]{5cm}
\centering
\includegraphics [scale=0.38]{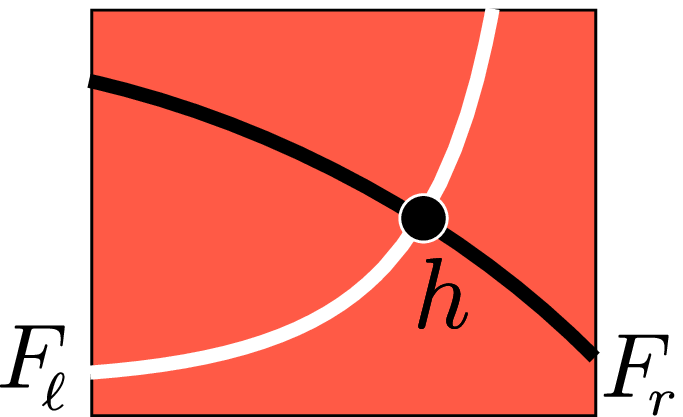}
\caption{\small A hyperbonode.}
\label{hyperbonode}
\end{minipage}
\hspace{1cm}
\begin{minipage}[t]{6cm}
\includegraphics [scale=0.38]{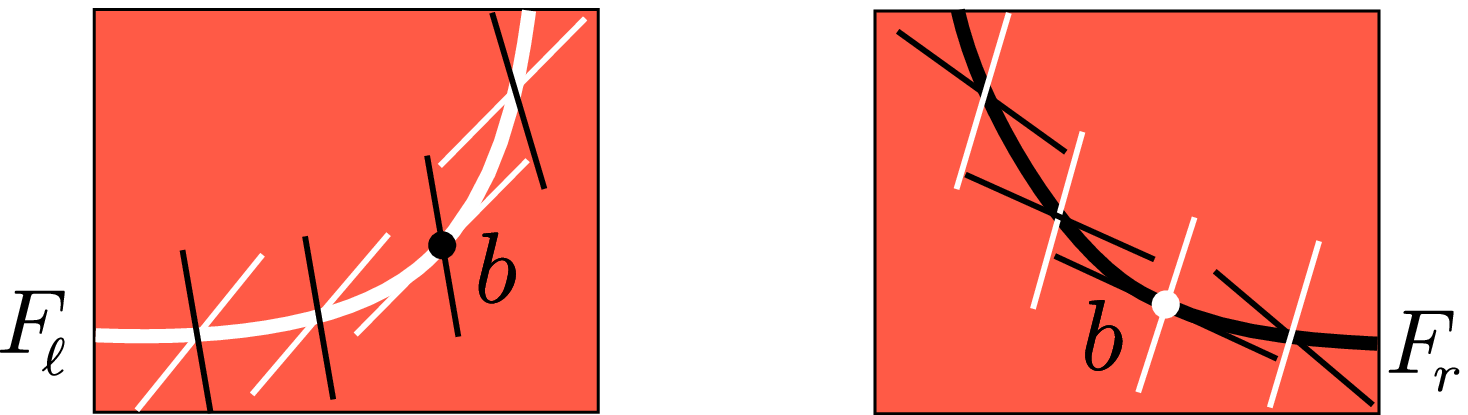}
\caption{\small Left and right biflecnodes.}
\label{biflecnodes}
\end{minipage}
\end{figure}

Ellipnodes and hyperbonodes are crucial in the metamorphosis of the parabolic curve of
evolving smooth surfaces, and in the metamorphosis of wave fronts occurring in generic
$1$-parameter families \cite{Uribeevolution}\,:
{\itshape any Morse metamorphosis of the parabolic curve takes place at an ellipnode (or hyperbonode)
  which is replaced by a hyperbonode (resp. by an ellipnode)} - Fig.\,\ref{metamorphosis}-left. 
    
  {\itshape Any $A_3$-metamorphosis of a wave front, where two swallowtail points are
    born or die, take place at an ellipnode (or hyperbonode) which is replaced by a
    hyperbonode (resp. by an ellipnode)} - Fig.\ref{metamorphosis}-right. 

\begin{figure}[h] 
\centering
\includegraphics[scale=0.4]{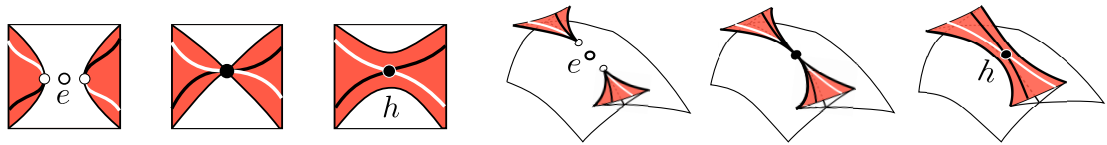}
\caption{\small Examples where an ellipnode is replaced by a hyperbonode or vice-versa.}
\label{metamorphosis}
\end{figure}

In this paper we introduce a geometric invariant for hyperbonodes as the cross-ratio determined by the
asymptotic lines and the tangents to the two branches of the flecnodal curve (it is in fact
a contact invariant).
This invariant describes some basic features of hyperbonodes and its sign distinguishes
the hyperbonodes that are born or die at a creation/annihilation transition.
We use this invariant to introduce an index for hyperbonodes, with values $+1$ or $-1$. Then we prove that
{\itshape the sum of the indices of the hyperbonodes inside a connected component $H$ of the hyperbolic
  domain remains constant in any $1$-parameter family of surfaces, provided that the topological type of $H$ does
not change} (Theorem\,\ref{index-invariance}). A similar theorem is proved for ellipnodes. 

We provide formulas for the invariant and for the index of a hyperbonode which do not depend on any
normal form (Theorems\,\ref{formula-invariant} and \ref{formula-index}).

\section{Cr-Invariant: Definition and Basic Properties}

An \textit{asymptotic curve} is an integral curve of a field of asymptotic directions.
\medskip

\noindent 
\textbf{\small Left and right asymptotic and flecnodal curves}. Fix an orientation in the
projective space $\RP^3$ (or in $\R^3$). The two asymptotic curves passing through a point of the
hyperbolic domain of a generic smooth surface are distinguished: one twists like a left screw
and the other like a right screw. More
precisely, a regularly parametrised smooth curve is said to be a \textit{left} (\textit{right})
\textit{curve} if its first three derivatives at each point form a negative (resp. a positive) frame.
\medskip

We need the following basic, but important, fact on surfaces (cf. \cite{Uribeinvariant}):  
\begin{fact*}
At each hyperbolic point one asymptotic curve is left and the other is right.
\end{fact*}

The respective tangent lines $L_\ell$, $L_r$ are called \textit{left} and \textit{right asymptotic lines}.  
\medskip

\textit{This left-right distinction does not depend on the orientability of the surface.} 

\begin{definition*}
The \textit{left} (\textit{right}) \textit{flecnodal curve} $F_\ell$ (resp. $F_r$) of a surface $S$ consists of
the points for which the left (resp. right) asymptotic line is over-osculating.
\end{definition*}

\begin{remark*}
  The transverse intersections of the left and right flecnodal curves (that is, the hyperbonodes)
  are stable; but the self-intersections of one branch of the flecnodal curve (left or right) are
  unstable \cite{Uribeevolution}. 
\end{remark*}

\noindent
\textbf{\small The projective invariant}. At a hyperbonode the two asymptotic lines $L_\ell$, $L_r$ and 
the tangent lines to the left and right flecnodal curves, which we note $L_{F_\ell}$ and $L_{F_r}$, 
determine a cross-ratio, which is a projective invariant\,:  
\medskip

\noindent
\textbf{\small Cr-invariant}. 
  The \textit{cr-invariant} $\rho(h)$ of a hyperbonode (hyperbolic projective umbilic) is the cross-ratio of the
  lines $L_{F_\ell}$, $L_r$, $L_{F_r}$, $L_\ell$\,:
  \[\rho(h):=(L_{F_\ell}, L_r, L_{F_r}, L_\ell).\] 

\begin{property}
  The cr-invariant does not depend on the chosen orientation of $\RP^3$.
  
  {\rm Indeed, a change of orientation permutes left and right, but
    $(L_{F_r}, L_\ell, L_{F_\ell}, L_r)=(L_{F_\ell}, L_r, L_{F_r}, L_\ell)$.}
\end{property}

\begin{property}
  By its definition, the value of $\rho$ is negative if $F_\ell$ and $L_r$
  locally separate $F_r$ from $L_\ell$ {\rm (fig.\,\ref{sign-invariant}).}
\end{property}

\begin{figure}[h] 
\centering
\includegraphics[scale=0.39]{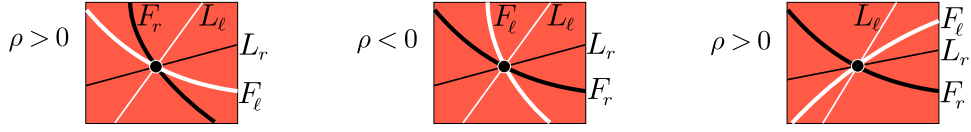}
\caption{\small Some generic positions of the asymptotic lines and flecnodal curves.}
\label{sign-invariant}
\end{figure}

\begin{property}\label{sign-distinguishes}
The sign of the cr-invariant distinguishes the hyperbonodes that take part in
a ``creation/annihilation'' transition {\rm (Fig.\,\ref{inv-distinguish})}. 
\end{property}

It follows because near the creation/annihilation moment (of two hyperbonodes) both branches of
the flecnodal curve are transverse to the asymptotic lines (Fig.\,\ref{inv-distinguish}). 

\begin{figure}[!ht]
\centering
\includegraphics[scale=0.42]{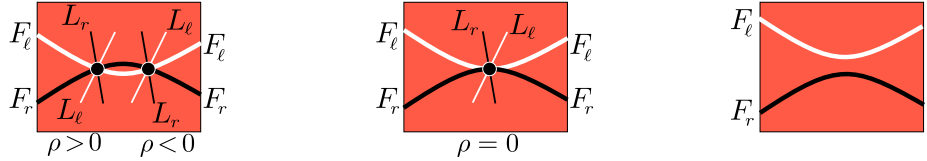}
\caption{\small A transition with ``creation/annihilation'' of two hyperbonodes.}
\label{inv-distinguish}
\end{figure}

\begin{property}
The sign of $\rho$ changes at 
a ``flec-hyperbonode'' transition {\rm (Fig.\,\ref{flec-hyperbonode})}. 
\end{property}

\begin{figure}[h] 
\centering
\includegraphics[scale=0.42]{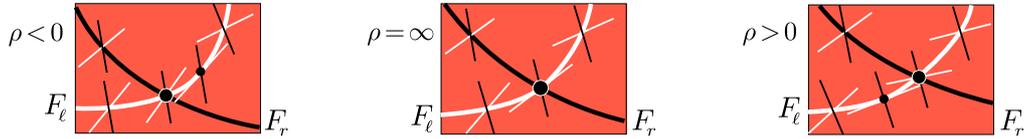}
\caption{\small A ``flec-hyperbonode'' transition where a biflecnode and a hyperbonode overlap.}
\label{flec-hyperbonode}
\end{figure}

To study the (unstable) hyperbonodes for which the cr-invariant change sign, $\rho=0$ and $\rho=\infty$,
one usually perturbs them inside a generic $1$-parameter family of smooth surfaces, where they become stable.
See Figs.\,\ref{inv-distinguish} and \ref{flec-hyperbonode}. 

\begin{property}\label{only-local-transitions}
  The only local transitions involving hyperbonodes, occurring inside the hyperbolic domain, correspond to
  the following values of $\rho$\,: 
\begin{enumerate}[label=\alph*), parsep=0.124cm, itemsep=0.0cm,topsep=0.2cm]\itemsep=0.0cm  
\item $F_\ell$ and $F_r$ are tangent $\iff$ $\rho=0$ (double hyperbonode)\,; 
\item $L_\ell$ is tangent to $F_\ell$ or $L_r$ is tangent to $F_r$ $\iff$ $\rho=\infty$ (flec-hyperbonode)\,.
\end{enumerate}
\end{property} 

\begin{remark*}
  The value $\rho=1$ corresponds to a non-generic hyperbonode where {\itshape $L_\ell$ is tangent to $F_r$
  or $L_r$ is tangent to $F_\ell$}, but this provides no relevant transition. 
\end{remark*}

\subsection{Some Expressions for the Invariant}\label{results}

\noindent
\textbf{\small Normal forms}. According to Landis-Platonova's Theorem \cite{Landis, Platonova} and
Ovsienko-Tabachnikov's Theorem \cite{Ovsienko-Tabachnikov}, the $4$-jet of a surface at a
hyperbonode can be sent by projective transformations to the
respective normal forms
\begin{equation*}\label{Landis normal form}
  z=xy+\frac{1}{3!}(ax^3y+bxy^3)+\frac{1}{4!}(x^4\pm y^4) \eqno(\text{L-P})
\end{equation*}
\begin{equation*}\label{Tabachnikov normal form}
  z=xy+\frac{1}{3!}(x^3y\pm xy^3)+\frac{1}{4!}(Ix^4+Jy^4)\eqno(\text{O-T})
\end{equation*}

\begin{remark*}
  These normal forms are equivalent for generic hyperbonodes.
  Indeed, writing $\x=x/u$, $\y=y/u$ and $\z=z/uv$ with $u^2=|a|$, $v^2=|b|$ we transform
  \[\z=\x\y+\frac{1}{3!}(a\x^3\y+b\x\y^3)+\frac{1}{4!}(\pm \x^4+\y^4) \quad \text{to} \quad
    z=xy\pm\frac{1}{3!}(x^3y\pm xy^3)+\frac{1}{4!}(Ix^4+Jy^4),\]
where $I=\pm |b/a^3|^{1/2}$ and $J=\pm |a/b^3|^{1/2}$ with the appropriate choice of signs. 

However, in (L-P) the biflecnodes are impossible, while in (O-T) the tangency of $L_r$ with $F_\ell$
(or of $L_\ell$ with $F_r$) is impossible. 
\end{remark*}

To encompass both normal forms we consider the
``prenormal'' form
\begin{equation*}\label{Ric-normal-form}
  z=xy+\frac{1}{3!}(ax^3y+bxy^3)+\frac{1}{4!}(Ix^4+Jy^4)\,.\eqno(\mathcal{H})
\end{equation*}

The following theorem implies that for $a=0$ ($b=0$) we get a hyperbonode where the left (resp. right)
asymptotic line is tangent to the right (resp. left) flecnodal curve; and for $I=0$ ($J=0$) we get a
hyperbonode where the right (resp. left) asymptotic line is tangent to the right  (resp. left) flecnodal
curve (i.e., a biflecnode overlaps with the hyperbonode - Fig.\,\ref{flec-hyperbonode}). 

\begin{theorem}
  Let $h$ be a hyperbonode, with cr-invariant $\rho$, of a generic smooth surface $S$.
  If we put the $4$-jet of $S$ (after projective transformations) in the ``prenormal'' form $(\mathcal{H})$, 
  then the resulting coefficients $a$, $b$, $I$, $J$ satisfy the relation
  \[1-\frac{ab}{IJ}=\rho\,.\]

  The respective relations satisfied by the normal forms {\rm (L-P)} and {\rm (O-T)} are 
\[1\mp ab=\rho \  \qquad \text{and} \ \qquad 1\mp\frac{1}{IJ}=\rho\,.\]
\end{theorem}

\begin{proof}
It is a direct corollary of Theorem\,\ref{formula-invariant} below. 
\end{proof}

\begin{remark*}
  In order that the normal form (L-P) represents only generic hyperbonodes,
  the restrictions $ab\neq 0$ and $ab\neq \pm 1$ have to be imposed. 
  Similarly, the conditions $IJ\neq 0$ and $IJ\neq \pm 1$ have to be imposed to (O-T). 
\end{remark*}

\noindent
\textbf{\small Formulas}. In order to get formulas for the cr-invariant, we shall identify the affine
chart of the projective space,
$\{[x : y : z : 1]\}\subset\RP^3$ with
$\R^3$, and present the germs of surfaces in $\R^3$ at the origin in Monge form \,$z=f(x,y)$\, with
\,$f(0,0)=0$\, and \,$df(0,0)=0$.
\medskip

\noindent
\textbf{\small Notation}. We shall express the partial derivatives of $f$ with numerical subscripts\,:
\[f_{ij}(x,y):=(\D^{i+j}f/\D x^i\D y^j)(x,y) \quad \text{and} \quad f_{ij}:=f_{ij}(0,0)\,.\]

\begin{theorem}[Proved in Section\,\ref{sect-proof-formula-invariant}]\label{formula-invariant}
  Let $h$ be a hyperbonode of a generic smooth surface. Taking the asymptotic lines as 
  coordinate axes, the cr-invariant of $h$ is given by
\begin{equation}\label{formula-invariant1}
  \rho(h)=1-\frac{(3f_{21}^2-2f_{11}f_{31})(3f_{12}^2-2f_{11}f_{13})}{4f_{11}^2f_{40}f_{04}}\,.
\end{equation}
\end{theorem}

Thus, the Monge form for the above special hyperbonodes ($\rho=0$ and $\rho=\infty$) satisfies
the following relations (taking the asymptotic lines as coordinates axes)\,:

\begin{corollary}
  At a hyperbonode the left and right branches of the flecnodal curve are tangent 
  (i.e. there is a double hyperbonode) if and only if
  \[4f_{11}^2f_{40}f_{04}-(3f_{21}^2-2f_{11}f_{31})(3f_{12}^2-2f_{11}f_{13})=0\,.\] 
\end{corollary}

\begin{proof}
One obtains this equality from formula \eqref{formula-invariant1} by putting $\rho(h)=0$. 
\end{proof}

\begin{corollary}
  At a hyperbonode $h$ an asymptotic tangent line has at least $5$-point contact with the surface
  (i.e. a biflecnode overlaps with $h$) if and only if $f_{40}f_{04}=0$. 
\end{corollary}

\begin{proof}
  At a hyperbonode we have $f_{30}=f_{03}=0$.
  The surface has exactly $4$-point contact with the $x$ and $y$ axes if and only if
  $f_{40}\neq 0$ and $f_{04}\neq 0$. The statement follows from formula \eqref{formula-invariant1} assuming that
  $\rho(h)=\infty$. 
\end{proof}

\noindent
\textbf{Theorem \ref{formula-invariant}'.}
  {\itshape If at a hyperbonode $h$ we take the diagonals $\,y=\pm x\,$ as asymptotic lines and we assume the cubic 
  terms of $f$ are missing for the chosen affine coordinate system, then}
\begin{equation}\label{formula-invariant1'}
\rho(h)=\frac{4[(f_{40}+3f_{22})(f_{04}+3f_{22})-
(f_{31}+3f_{13})(f_{13}+3f_{31})]}{(f_{40}+6f_{22}+f_{04})^2-16(f_{31}+f_{13})^2}\,.
\end{equation}  

\section{The index of a hyperbonode}

At any point of the left flecnodal curve, which is not a left biflecnode, the left asymptotic line has
exactly $4$-point contact with the surface. Thus at all such points the left asymptotic line locally
lies in one of the two sides of the surface - the same holds for the right flecnodal curve.
Then in a tubular neighbourhood of the flecnodal curve, the surface is locally cooriented by the corresponding
asymptotic line - this local coorientation changes only at the biflecnodes.

At a hyperbonode both asymptotic lines may locally lie in the same side of the surface, but also may
locally lie in opposite sides.  
\medskip

\noindent 
\textbf{\small Parity}. We say that a hyperbonode has \textit{parity} $\s(h)=+1$ if both asymptotic lines
locally lie on the same side of the surface and $\s(h)=-1$ 
otherwise.\footnote{The    
choice of a local coorientation of the surface $S$ provides a \textit{signature} to each hyperbonode 
(\cite{Ovsienko-Tabachnikov}): it is a pair of signs, $s=(s_r,s_\ell)$, which are $+$ (or $-$) if that 
local coorientation coincides (or not) with the coorientations given by the left and right asymptotic 
lines. Of course, $s$ depends on the chosen coorientation of $S$. 
The points of signature $s=(+,+)$ or $(-,-)$ have parity $\s=+1$, and those of signature 
$(+,-)$ or $(-,+)$ have parity $\s=-1$.}

\begin{remark*} 
The cr-invariant $\rho$ and the parity $\s$ are independent of the orientation of $\RP^3$.  
\end{remark*}

\noindent
\textbf{\small Index}.
  A generic hyperbonode is said to be \textit{positive} or of \textit{index} $+1$ if \,$\s(h)\rho(h)>0$, 
  and \textit{negative} or of  \textit{index} $-1$ if $\s(h)\rho(h)<0$, that is, 
  \[\ind(h)=\s(h)\,\sign(\rho(h))\,.\]

\begin{proposition}
In generic $1$-parameter families of smooth surfaces, the two hyperbonodes that collapse during a
creation/annihilation transition have opposite indices.  
\end{proposition}

\begin{proof}
  The cr-invariant has distinct sign at these hyperbonodes (Property\,\ref{sign-distinguishes}), and
  either both hyperbonodes are odd, $\s(h)=-1$, or both are even, $\s(h)=+1$, because there is no biflecnode
  between them. Then they have opposite indices. 
\end{proof}

\begin{proposition}
  When a hyperbonode undergoes a ``flec-hyperbonode'' transition
  its index does not change {\rm (Fig.\,\ref{flec-hyperbonode})}. 
\end{proposition}

\begin{proof}
  At a ``flec-hyperbonode'' transition both the parity and the cr-invariant of the hyperbonode
  change sign. Thus the index remains constant. 
\end{proof}

A corollary of these propositions is the following theorem\,: 

\begin{theorem}\label{index-invariance}
The sum of the indices of the hyperbonodes inside a connected component $H$ of the hyperbolic
domain remains constant in a continuous $1$-parameter family of surfaces, provided that the topological
type of $H$ does not change.
\end{theorem}

\begin{proof}
The only local transitions that involve hyperbonodes, occurring in the interior of the hyperbolic domain, are
the  creation/annihilation and the flec-hyperbonode transitions (Property \ref{only-local-transitions}).
\end {proof}

A consequence of the theorem of \cite{Uribeinvariant} (p.\,744) quoted in the introduction is the

\begin{corollary*}
  In a hyperbolic disc $D$ bounded by a closed parabolic curve, the sum of the indices of the
  hyperbonodes inside $D$ equals $1$. 
\end{corollary*}

\begin{theorem}
If the $4$-jet of $h$ is expressed in the ``prenormal form'' $(\mathcal{H})$, then
\[\ind(h)=\sign(IJ-ab)\,.\]
\end{theorem}
\begin{proof}
It is a direct corollary of Theorem\,\ref{formula-index} below.
\end{proof}

\begin{theorem}\label{formula-index}
If the asymptotic lines are the coordinate axes at a hyperbonode $h$, then 
  \[\ind(h)=\sign\left(4f_{11}^2f_{40}f_{04}-(3f_{21}^2-2f_{11}f_{31})(3f_{12}^2-2f_{11}f_{13})\right)\,.\]
\end{theorem}

\begin{proof}
The formula follows from Theorem\,\ref{formula-invariant} because $\s(h)=\sign(f_{40}f_{04})$. 
\end{proof}

\noindent
\textbf{Theorem \ref{formula-index}'.}
{\itshape If the asymptotic lines are the diagonals $y=\pm x$ at a hyperbonode $h$, then}
\[\ind(h)=\sign\left((f_{40}+3f_{22})(f_{04}+3f_{22})-(f_{31}+3f_{13})(f_{13}+3f_{31})\right)\,.\]

\noindent
\textit{Proof}. 
This follows from Theorem\,\ref{formula-invariant}' because 
$\s(h)=\sign\left((f_{40}+6f_{22}+f_{04})^2-16(f_{31}+f_{13})^2\right)$. 
Indeed, the parity is the sign of $\D^4_{v^+}f\cdot\D^4_{v^-}f$ at the origin, with $v^+=(1,1)$, $v^-=(1,-1)$, and 
 \[\D^4_{v^+}f=(f_{40}+6f_{22}+f_{04})+4(f_{31}+f_{13})\,, \qquad 
\D^4_{v^-}f=(f_{40}+6f_{22}+f_{04})-4(f_{31}+f_{13})\,. \eqno{\square}\]


\section{Some Remarks and a Problem}\label{remarks}

\begin{remark}
  In \cite{Kazarian-Uribe}, hyperbonodes and ellipnodes are characterised as the singular points
  of an intrinsic field of lines on the surface; then an index is defined at these points.

In the case of hyperbonodes, our index $\ind(h)$ coincides with the index of \cite{Kazarian-Uribe},  
providing a new characterisation of it. In  \cite{Kazarian-Uribe}, it was proved that\,:
\smallskip

\noindent
{\itshape The sum of indices of the hyperbonodes lying inside a connected component of the
  hyperbolic domain equals the Euler characteristic of that domain}. 
\end{remark}

\begin{remark}
In the case of ellipnodes, the index of \cite{Kazarian-Uribe} has fractional values $+1/3$ or $-1/3$. 
A {\em godron} is a parabolic point at which the unique (but double) asymptotic line is tangent
to the parabolic curve. The godrons of generic surfaces have an intrinsic index with value $+1$ or
$-1$ (cf. \cite{Dima, Uribeinvariant}).

A godron, hyperbonode or ellipnode is said to be \textit{positive} (\textit{negative}) if its index is
positive (resp. negative). The following theorem was proved in \cite{Kazarian-Uribe}\,: 

Let $H$ be a connected component of the hyperbolic domain and $E$ a connected component of the elliptic domain.
\begin{theorem*}
Write $\#e(E)$ for the number of positive ellipnodes of $E$ minus the negative ones; $\#g(E)$ for the number of
positive godrons on $\D E$ minus the negative ones; $\#h(H)$ for the number of positive hyperbonodes of $H$
minus the negative ones; and $\#g(H)$ for the number of positive godrons on $\D H$ minus the negative ones.
Then
\[\#h(H)+\#g(\D H)=3\,\chi(H) \ \quad\text{and}\quad \ \#e(E)-\#g(\D E)=3\,\chi(E)\,.\]
\end{theorem*}
\end{remark}

\begin{remark}
  Unfortunately, the results of the present paper and those of \cite{Kazarian-Uribe} provide no help to prove or disprove
  Ovsienko-Tabachnikov's conjecture because a compact smooth hyperbolic surface has zero Euler characteristic.  
\end{remark}

\begin{problem}
  Landis-Platonova's and Ovsienko-Tabachnikov's normal forms contain two moduli.
  Then, to determine the $4$-jet of the surface at a hyperbonode the geometric invariant $\rho$
  does not suffice. The problem is to find a second geometric invariant (in order to determine that
  $4$-jet by two geometric invariants). 
\end{problem}

\section{The Cross-Ratio Invariant for Ellipnodes}\label{sect-invariant-ellipnodes}
At an elliptic point of a generic surface there is a pair of complex conjugate asymptotic lines 
(the zeros of the quadratic form in the Monge form). 
For example, if $f(x,y)=x^2+y^2+\ldots$ the asymptotic lines at the origin are 
$L: \ y=ix$ and $\overline{L}: \y=-ix$. 
Near an ellipnode there is a corresponding pair of complex conjugate flecnodal curves (consisting of 
the complex points where an asymptotic line is also a line of zeros of the cubic term of $f$). 
The cross-ratio defined by the asymptotic lines $L_i, \overline{L}$ with 
the tangent directions of the flecnodal curves $F,\overline{F}$ is a real number:  
\[\rho(e):=\left(\,L_{\overline{F}},\, L,\, L_F,\, \overline{L}\,\right).\] 

Express the surface in Monge form $z=f(x,y)$ with quadratic part $Q=\frac{\a}{2}(x^2+y^2)$ and 
assume that (for the chosen affine coordinate system) the cubic terms of $f$ are missing. 
\begin{theorem}\label{local-index-e}
For an ellipnode $e$ of a generic smooth surface the cr-invariant is given by
\begin{equation}\label{formula-invariant1-e}
\rho(e)=\frac{4[(f_{40}-3f_{22})(f_{04}-3f_{22})-(f_{31}-3f_{13})(f_{13}-3f_{31})
]}{(f_{40}-6f_{22}+f_{04})^2+16(f_{31}-f_{13})^2}\,.
\end{equation}
{\rm (Compare this formula with formula\,\eqref{formula-invariant1'} of Theorem\,\ref{formula-invariant}'.)}
\end{theorem}

As in Property\,\ref{sign-distinguishes}, the cr-invariant of an ellipnode vanishes 
if the complex conjugate flecnodal curves are tangent. Hence the sign of the cr-invariant 
distinguishes the ellipnodes that take part in a ``creation/annihilation'' transition 
{\rm (Fig.\,\ref{inv-distinguish-e})}. In fact, in the case of ellipnodes the sign of the cross-ratio 
invariant behaves as an index (the value of the index being defined up to a factor)\,:  
\smallskip

\noindent 
\textbf{\small Positive and Negative Ellipnodes}. An ellipnode is said to be \textit{positive} 
or \textit{negative} if its cross-ratio invariant is respectively positive or negative.  

\begin{corollary}\label{corollary-doule-ellipnode}
  An ellipnode is double - it is a creation/annihilation transition of two ellipnodes of opposite sign - if and only if
  \[(f_{40}-3f_{22})(f_{04}-3f_{22})-(f_{31}-3f_{13})(f_{13}-3f_{31})=0\,.\] 
\end{corollary}

\begin{figure}[!ht] 
\centering
\includegraphics[scale=0.42]{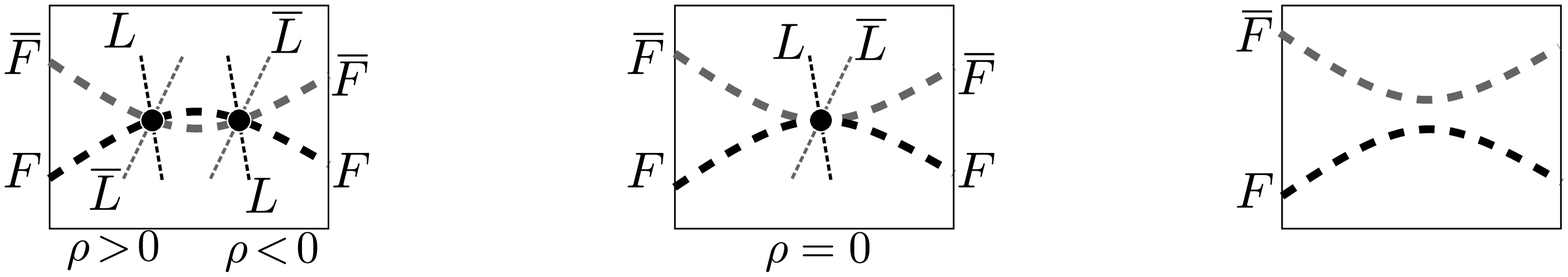}
\caption{\small``Creation/annihilation'' of $2$ ellipnodes at a tangency of $2$ complex conjugate flecnodal curves.}
\label{inv-distinguish-e}
\end{figure}

In generic $1$-parameter families of smooth surfaces, the cross-ratio invariant of an ellipnode 
never takes the value $\infty$ because the vanishing of the denominator of \eqref{formula-invariant1-e} 
would need two conditions. 
A consequence of this fact and of Corollary\,\ref{corollary-doule-ellipnode} is the following theorem\,: 

\begin{theorem}\label{ellipnodes-index-invariance}
The algebraic sum of the ellipnodes inside a connected component $E$ of the elliptic
domain remains constant in a continuous $1$-parameter family of surfaces, provided that the topological
type of $E$ does not change.
\end{theorem}

\begin{proof}
The only local transition that involves ellipnodes, occurring in the interior of the elliptic domain, is
the creation/annihilation of two ellipnodes of opposite sign.
\end {proof}


\section{Proof of Theorem\,\ref{formula-invariant}}\label{sect-proof-formula-invariant}

We take the asymptotic lines at $h$ as coordinate axes, that is, $f_{20}=f_{02}=0$, and  
we use the fact that $h$ is a hyperbonode, that is, $f_{30}=f_{03}=0$.

\begin{lemma*}[proved below]
The tangent lines to the right and left flecnodal curves at $h$ have the respective equations 
\[2f_{11}f_{40}\,x-(3f_{21}^2-2f_{11}f_{31})\,y\,=\,0\,. \eqno{(L_{F_r})}\]
\[-(3f_{12}^2-2f_{11}f_{13})\,x+2f_{11}f_{04}\,y\,=\,0\,. \eqno{(L_{F_\ell})}\]
\end{lemma*}

\noindent 
\textit{Proof of Theorem} \ref{formula-invariant}. 
  To compute $\rho(h)$, we shall use the intersection points of the lines $L_{F_\ell}$, $L_r$, $L_{F_r}$, $L_\ell$
  with the line of equation $y=1$, which is parallel to $L_r$.
  The $x$-coordinates of the respective intersection points are (by the Lemma)\,:
\[x_{F_\ell}=\frac{2f_{11}f_{04}}{(3f_{12}^2-2f_{11}f_{13})}; \qquad x_{L_r}=\infty; \qquad
x_{F_r}=\frac{(3f_{21}^2-2f_{11}f_{31})}{2f_{11}f_{40}}; \qquad x_{L_\ell}=0\,.\]

We have that $\rho(h)=(x_{F_\ell}-x_{F_r})/(x_{F_\ell}-x_{L_\ell})$, because $x_{L_r}=\infty$. Therefore, 
\[\rho(h)=1-\frac{(3f_{21}^2-2f_{11}f_{31})(3f_{12}^2-2f_{11}f_{13})}{4f_{11}^2f_{40}f_{04}}\,. \eqno{\square} \]

Before to prove the Lemma, we very briefly explain the geometry that justify the needed calculations.
For more a complete explanation see \cite{Uribeevolution}. 

\subsection{Flecnodal Curve and Geometry in the Space of $1$-Jets}

Assume our surface $S$ is locally given in Monge form $z=f(x,y)$, where our hyperbonode $h$
is the origin. 
At each point of $S$ the pair of asymptotic lines is determined by the kernel
of the second fundamental form of $S$.
This means that each solution of the implicit differential equation (IDE)
\begin{equation}\label{asymptote-equation}
  f_{20}(x,y)dx^2+2f_{11}(x,y)dxdy+f_{02}(x,y)dy^2=0\,.
\end{equation}
is the image of an asymptotic curve by the projection $(x,y,z)\to(x,y)$. 
Moreover, {\itshape the image of the flecnodal curve $F$ under the same projection $(x,y,z)\to(x,y)$ is
  the curve $\widehat{F}$ formed by the inflections of the solutions of} \eqref{asymptote-equation}
(cf. \cite{Uribetesis, Uribeevolution}).

Since our hyperbonode $h$ is the origin of the $xy$-plane, which is the tangent plane to $S$
at $h$, the curves $F$ and $\widehat{F}$ have the same tangent lines.
Therefore, we only need to find the tangent lines to the left and right branches
of $\widehat{F}$. 

Putting $p=dy/dx$, the IDE \eqref{asymptote-equation} is ``materialised'' by the surface
$\mathcal A\subset J^1(\R,\R)$ of equation 
\[a^f(x,y,p):=f_{20}(x,y)+2f_{11}(x,y)p+f_{02}(x,y)p^2=0\,.\]
{\footnotesize
\begin{remark*}
  The surface $\A$ does not take into account the asymptotic lines parallel to the $y$ direction
  because this direction corresponds to $p=\infty$.   
\end{remark*} }

The space of $1$-jets $J^1(\R,\R)$ (with coordinates $x,y,p$) has a natural \textit{contact structure}:
it is the field of \textit{contact planes} on $J^1(\R,\R)$ which are the kernels of the
$1$-form $\a=dy-pdx$.  At each generic point of $\A$ the contact plane and the tangent pane to $\A$
determine a \textit{characteristic line} tangent to $\A$. 

The surface $\A$ is foliated by the integral curves of the characteristic lines of $\A$,
whose projections to the $xy$-plane are the solutions of the IDE $a^f(x,y,p)=0$.

In \cite{Uribetesis, Uribeevolution}, it is proved that the curve of inflections $\widehat{F}$
is the image by the natural projection $\pi:\A\subset J^1(\R,\R)\to\R^2$, $(x,y,p)\mapsto (x,y)$, of the 
curve $\mathcal{F}\subset\A$ formed by the fold points of the Legendre dual projection 
$\pi^\vee:\A\ni J^1(\R,\R)\mapsto(p,px-y)$.
That is, $\mathcal{F}\subset\A$ is the apparent contour of $\A$ by the projection $\pi^\vee:\A\to(\R^2)^\vee$

According to \cite{Uribetesis, Uribeevolution}, the apparent contour $\mathcal{F}\subset\A$
is the  intersection of $\mathcal A$ with the surface $\mathcal I\subset J^1(\R,\R)$ given
locally by the equation 
\[I(x,y,p):=a^f_x(x,y,p)+pa^f_y(x,y,p)=0\,.\]

Therefore, the ``curve of inflections'' $\widehat{F}$ is (locally) the projection to the $xy$-plane,
along the $p$-direction, of the intersection $\A\cap\mathcal I$. Now we can prove the Lemma.

\subsection{Proof of the Lemma}

The right flecnodal curve is locally obtained from the intersection $\A\cap\mathcal I$ near the origin.
Therefore, the tangent line to the right flecnodal curve at $h$ is the projection of the intersection
line of the tangent planes to the surfaces $\mathcal A$ and $\mathcal I$ at $(0,0,0)$.
These planes are provided by the differentials of $a^f$ and $I$
at the origin:
\[da^f_{|_{\bar{0}}}=f_{30}dx+f_{21}dy+2f_{11}dp \qquad \text{and} \qquad dI_{|_{\bar{0}}}=f_{40}dx+f_{31}dy+2f_{21}dp\,.\]

Using the condition $f_{30}=0$, we get the equations of the tangent planes\,:
\[f_{21}y+2f_{11}p=0  \qquad \text{and} \qquad f_{40}x+f_{31}y+2f_{21}p=0\,.\] 

Then the tangent line $L_{F_r}$ to the right flecnodal curve at the origin has equation 
\[L_{F_r}: \qquad 2f_{11}f_{40}\,x-(3f_{21}^2-2f_{11}f_{31})\,y\,=\,0\,.\]

To get the tangent line to left flecnodal curve at the origin, in which the $y$ axis is the left
asymptotic line, we make similar computations with the equations
\[\hat{a}^f(x,y,q):=f_{20}q^2+2f_{11}q+f_{02}=0 \qquad \text{and} \qquad \hat{a}^f_y(x,y,q)+q\hat{a}^f_x(x,y,q)=0.\]
Then we get the equation of the tangent line $L_{F_\ell}$to the left flecnodal curve\,:
\[L_{F_\ell}: \qquad 2f_{11}f_{04}\,y-(3f_{12}^2-2f_{11}f_{13})\,x\,=\,0\,. \eqno{\square}\]

\paragraph{On the Proof of Theorems\,\ref{formula-invariant}' and \ref{local-index-e}.}
The proofs of the Theorems\,\ref{formula-invariant}' and \ref{local-index-e} follow exactly the same 
pattern of the proof of Theorem\,\ref{formula-invariant}. 
Moreover, in Theorems\,\ref{formula-invariant}' and \ref{local-index-e}, the Monge forms differ only 
in one sign of the quadratic term (being respectively $x^2-y^2$ and $x^2+y^2$), 
so that the calculations involved in their proofs are almost identical. 

To compute the cross-ratio in the elliptic case (Theorem\,\ref{local-index-e}),  we take the tangent 
plane to $S$ at an ellipnode $e$ where we find the points of intersection of the line of equation $y=-ix+2i$ 
with the (complex conjugate) asymptotic lines ($L$: $\,ix-y=0$\,; \,$\overline{L}$: $\,ix+y=0$;) and with the 
tangent lines to the (complex conjugate) flecnodal curves 
($L_F$: $\,ax+by=0$\,; \,$L_{\overline{F}}$: $\,\bar{a}x+\bar{b}y=0$;), 
where $a=(f_{40}-3f_{22})+i(3f_{31}-f_{13})$ and $b=(f_{31}-3f_{13})-i(f_{04}-3f_{22})$. 

The $x$-coordinates of those four intersection points are 
\[x_{L}=1; \qquad x_{\overline{L}}=\infty; \qquad
x_{F}=\frac{2ib}{bi-a}; \qquad x_{\overline{F}}=\frac{2i\bar{b}}{\bar{b}i-\bar{a}}\,.\]

Hence $\,\rho(e)=(x_{F}-x_{\overline{F}})/(x_{F}-x_{L})$ because $x_{\overline{L}}=\infty$. Therefore, 
\[\rho(e)=\frac{4[(f_{40}-3f_{22})(f_{04}-3f_{22})-(f_{31}-3f_{13})(f_{13}-3f_{31})
]}{(f_{40}-6f_{22}+f_{04})^2+16(f_{31}-f_{13})^2}\,. \eqno{\square} \]


\bigskip

\noindent
{\scshape Ricardo Uribe-Vargas, \\
Institut de Math\'ematiques de Bourgogne, UMR 5584 CNRS. \\
Universit\'e Bourgogne Franche-Comt\'e, F-21000 Dijon, France}. \\
r.uribe-vargas@u-bourgogne.fr


{\footnotesize 
\begin{thebibliography}{99}

\bibitem {Arnold-RPCS} \textbf{Arnold V.I.},
\textit{``Remarks on the Parabolic Curves on Surfaces and on the
    Higher-Dimensional M\"obius-Sturm Theory,''} Funct. Anal. Appl. \textbf{31} (1997) 227-239.

\bibitem {Arnold-TPTAC} \textbf{Arnold V.I.},
  \textit{``Topological Problems of the Theory of Asymptotic Curves,''}
  Proc. Steklov Inst. Math. \textbf{225} (1999) 5-15.
  
\bibitem {Cayley} \textbf{Cayley A.}, 
\textit{On the singularities of surfaces}, Camb. Dublin Math. J.\,\textbf{7} (1852), 166-171.
  
\bibitem {Chekanov-Pushkar} \textbf{Chekanov Yu. V.}, \textbf{Pushkar P. E.}, 
  \textit{``Combinatorics of Fronts of Legendrian Links and the Arnol’d 4-Conjectures,''}
  Russ. Math. Surv. \textbf{60} (2005) 95-149.
  
\bibitem {Garcia-Freitas} \textbf{Freitas B.}, \textbf{Garcia R.}, 
\textit{Inflection points of hyperbolic tori of $\sph^3$}, 
Quart. J. Math. 00 (2018), 1-20; doi:10.1093/qmath/hax058

\bibitem {Goryunov84} \textbf{Goryunov V. V.},
  \textit{Singularities of projections of complete intersections}.
  J. Sov. Math. \textbf{27} (1984) 2785-2811. 

\bibitem {Kazarian-Uribe} \textbf{Kazarian M.}, \textbf{Uribe-Vargas R.},
  \textit{Characteristic Points, Fundamental Cubic Form
    and Euler Characteristic of Projective Surfaces}. Preprint 2013. 
To appear in Moscow Mathematical Journal. 

\bibitem {Landis} \textbf{Landis E.E.}, {\em Tangential singularities},
Funct. Anal. Appl. \textbf{15}:2 (1981), 103-114. 


\bibitem {Ovsienko-Tabachnikov} \textbf{Ovsienko V.}, \textbf{Tabachnikov S.}, 
  \textit{Hyperbolic Carathéodory Conjecture}. Proc. of Steklov Inst. of Mathematics,
  \textbf{258} (2007) 178-193.

 \bibitem {Dima} \textbf{Panov D.A.}, {\em Special Points of Surfaces in the 
Three-Dimensional Projective Space}, 
  Funct. Anal. Appl. \textbf{34}:4 (2000) 276-287.  

\bibitem {Platonova} \textbf{Platonova O.A.}, {\em Singularities of the mutual position 
of a surface and a line}, Russ. Math. Surv. \textbf{36}:1 (1981) 248-249. Zbl.458.14014.

\bibitem {Salmon} \textbf{Salmon G.}, {\em A Treatise in Analytic 
    Geometry of Three Dimensions}, Chelsea Publ. (1927).

\bibitem {Toru} \textbf{Sano H.}, \textbf{Kabata Y.}, \textbf{Deolindo Silva J.}, \textbf{Ohmoto T.},
  \textit{Classification of Jets of Surfaces in Projective $3$-Space Via Central Projection},
  Bull. Math. Braz. Soc. (2017), DOI 10.1007/s00574-017-0036-x. 
  
\bibitem {Uribetesis} \textbf{Uribe-Vargas R.}, {\em Singularit\'es symplectiques et de 
contact en g\'eom\'etrie diff\'erentielle des courbes et des surfaces}, 
  PhD. Thesis, Universit\'e Paris 7, 2001, (In English).

\bibitem {Uribeevolution} \textbf{Uribe-Vargas R.}, \textit{Surface Evolution, Implicit 
Differential Equations and Pairs of Legendrian Fibrations}, Preprint (2002). 
An improved version has been submitted for publication (2019). 

\bibitem {Uribeinvariant} \textbf{Uribe-Vargas R.}, \textit{A Projective Invariant for
Swallowtails and Godrons, and Global Theorems on the Flecnodal Curve}, 
  Moscow Math. Jour. \textbf{6}:4 (2006) 731-772.

\bibitem {Wilczynski} \textbf{Wilczynski E.\,J.},
  \textit{``Projective Differential Geometry of Curved Surfaces. I-V,''} Trans. Am. Math.
  Soc. \textbf{8} (1907) 233-260; \textbf{9} (1908) 79-120, 293-315; \textbf{10}, 176-200, (1909) 279-296.

  
\end{thebibliography}
}
\end{document}